\documentclass{amsart}

\usepackage{amsmath}
\usepackage{amssymb}
\usepackage{color}
\usepackage{enumerate}
\usepackage{amsthm}
\usepackage{bbm}
\usepackage{listings}
\usepackage[all]{xy}
\usepackage{longtable}
\usepackage{graphicx}

\newtheorem{theorem}{Theorem}[section]
\newtheorem{proposition}[theorem]{Proposition}
\newtheorem{lemma}[theorem]{Lemma}
\newtheorem{corollary}[theorem]{Corollary}
\theoremstyle{definition}
\newtheorem{definition}[theorem]{Definition}
\theoremstyle{remark}
\newtheorem{remark}[theorem]{Remark}
\newtheorem{example}[theorem]{Example}

\newcommand{\rank}{\operatorname{rank}}

\newcommand{\D}{\operatorname{d}\!}
\newcommand{\CC}{\mathbb{C}}
\newcommand{\QQ}{\mathbb{Q}}

\newcommand{\NN}{\mathbb{N}}
\newcommand{\RR}{\mathbb{R}}

\newcommand{\Sing}{\operatorname{Sing}}
\newcommand{\corank}{\operatorname{corank}}
\newcommand{\id}{\operatorname{id}}
\newcommand{\Inst}{\operatorname{Inst}}

\newcommand{\sign}{\operatorname{sign}}
\newcommand{\Alt}{\operatorname{Alt}}

\date{}
\title{Kato-Matsumoto-type results for disentanglements}
\author[G. Pe\~nafort Sanchis]{G. Pe\~nafort Sanchis$^\dagger$}
\author[M. Zach]{M. Zach$^\heartsuit$}
\address[$^\dagger$]{Basque Center for Applied Mathematics\\
Mazarredo Zumarkalea, 14\\ 48009 Bilbo, Spain}
\address[$\heartsuit$]{Institut f\"ur Mathematik \\
FB 08 - Physik, Mathematik und Informatik \\
Johannes Gutenberg-Universit\"at \\
Staudingerweg 9, 4. OG \\
55128 Mainz, Germany}

\begin{document}

\maketitle

\begin{abstract}
  \noindent
  We consider the possible 
  disentanglements of holomorphic map germs $f \colon (\CC^n,0) \to (\CC^N,0)$,
  $0<n < N$, with nonisolated locus of instability $\Inst(f)$. The aim is to achieve lower 
  bounds for their (homological) connectivity in terms of $\dim \Inst(f)$. 
  Our methods  apply in the case of corank $1$.\\

  \noindent
  MSC classification: \subjclass{58K15, 58K60, 32S30}.  \\
  \noindent
  Keywords: \keywords{Vanishing homology, disentanglements, connectedness}.
\end{abstract}

\section{Introduction}

Let $g\colon (\CC^{n+1},0)\to (\CC,0)$ be a holomorphic function germ defining an isolated hypersurface singularity $Z=g^{-1}(\{0\})$. The Milnor fibre $M$ has the homotopy type of a wedge of $n$-dimensional real spheres, see \cite{Milnor68}. In particular, the reduced homology of $M$ is concentrated 
in degree $n$. Since its appearance, this result has been subject to several generalizations and modifications. We are interested in the connections between two modifications of the original setup, which we refer to as the \emph{fibration} setup and the \emph{parametrization} setup.

The fibration setup generalizes the original one in two ways. In the first one, $Z$ is a higher codimension complete intersection with isolated singularity. It is a result of Hamm \cite{Hamm72} that $M$ still has the homotopy type of a bouquet of spheres of real dimension $\dim Z$. Alternatively, one can ask $Z=g^{-1}(\{0\})$ to be a hypersurface but let the singular locus have dimension $d$. Kato and Matsumoto showed in \cite{KatoMatsumoto75} that the Milnor fibre of $g$ is at least $(n-d-1)$-connected.
The fact that this holds also for non-isolated complete intersection singularities has been part of the folklore. For lack of a reference, we give a proof in Section \ref{sec:Non-IsoCompleteInter}.

In the parametrization setup, the space $Z$ is the image of a finite holomorphic map germ $f\colon (\CC^n, 0) \to (\CC^N,0)$, and one would like to understand the changes in topology after perturbing $f$ to a stable map. The map germs with isolated instabilities are the $\mathcal A$-finite map germs, and the image of a good representative of a stabilization of $f$ is called a disentanglement.  Thinking of the fibrations and parametrization settings as different instances of the general  deformation theory reveals interesting connections. Milnor's original setting, dealing with isolated hypersurface singularities, corresponds to that of $\mathcal A$-finite parametrization $f\colon (\CC^n,0)\to (\CC^{n+1},0)$.  D. Mond showed in \cite{Mond91} that the disentanglement has the homotopy type of a bouquet of $n$-spheres. The number of spheres in the bouquet is called the \emph{image Milnor number} $\mu_I(f)$. Another illustrating example is Mond's conjecture. It claims the inequality $\mu_I(f)\geq \mathcal A_e{\rm -codim}(f)$ -- motivated, by analogy, from the trivial inequality $\mu\geq \tau$ for the Milnor and Tjurina numbers of hypersurfaces with isolated singular locus. Moving to higher codimension, the setting of Hamm's result corresponds to that of $\mathcal A$-finite map germs $(\CC^n,0)\to (\CC^{N},0)$ with $N>n+1$. The degrees containing the non-trivial rational homology were determined for corank one germs by Goryunov and Mond in \cite{GoryunovMond93}; the result being generalized to arbitrary corank by Houston in \cite{Houston:1997}. In this work, we show the analog of Kato and Matsumoto's result for disentanglements.

 The restriction we have 
to make is to consider only the case of corank $1$ germs with 
multiple point spaces of the expected dimension. 
In this situation the latter are complete intersections  
and we can study their  induced deformations 
by means of the generalized Kato-Matsumoto Theorem.
Then, we apply a spectral sequence argument due to Goryunov and Mond \cite{GoryunovMond93} 
to compute the rational homology of the disentanglement. 

Just as in \cite{GoryunovMond93}, passing through the spectral sequence  looses 
the information about the homotopy type of the disentanglement as well 
as its integer homology groups. 

\subsection*{Acknowledgements}
The first author was partially supported by the ERCEA 615655 NMST Consolidator Grant and 
by the Basque Government through the BERC 2014-2017 program and by Spanish Ministry of Economy and
Competitiveness MINECO: BCAM Severo Ochoa excellence accreditation SEV-2013-0323.
The second author wishes to thank for the kind hospitality during a stay at the University of Valencia, during which the work on this subject was initiated.

\section{Overview and results}We briefly outline the main theorem of this article and recall 
the common definitions of the objects involved.
Consider a finite, holomorphic map germ 
\[
  f \colon  (\CC^n,0) \to (\CC^N,0).
\]
Let $( \Inst(f),0 ) \subset ( \CC^N,0)$ be the instability locus of $f$ (see Definition \ref{defStability}) and let $d = \dim (\Inst(f),0)$. 
We will mostly be concerned with non-isolated instability, i.e. 
the case $d>0$. 
We may choose an \textit{unfolding}
\[
  F \colon  (\CC^n,0) \times (\CC,0) \to (\CC^N,0)\times (\CC,0), \quad 
  (x,u) \mapsto (f_u(x),u)
\]
of $f = f_0$
and a \textit{good, proper representative} 
$F \colon  X \times U \to Y \times U$ thereof in a sense made precise below (cf. Definition 
\ref{def:StabilizationAndStablePerturbation} and Section \ref{sec:FibrationTheorems}). 
Using stratification theory and Thom's isotopy lemmas one can show that 
the image $F(X)$ is a topological fiber bundle 
with fiber $f_u(X)$ over all $u\neq 0$ in some small disc $D \subset U$ around 
the origin. 
In case that $F$ is a $C^0$-stabilization (see Definition \ref{defStability}) 
of $f$, we call the space 
$Z_u:=f_u(X)$ a \textit{disentanglement} of $f$. However, since our 
considerations also involve multiple point spaces and their induced deformations,
we will have to chose the representatives accordingly.
Also note that, contrary to the case of isolated 
instability, our notion of Milnor fiber really depends on the chosen unfolding $F$ and not 
only on the map germ $f$. 

\begin{theorem}
  In the setup described above,
  the reduced homology with rational coefficients 
  $\tilde H_q( Z_u, \QQ)$ 
  of a disentanglement $Z_u = f_u(X)$ of a dimensionally correct 
  corank $1$ map germ 
  $f \colon  (\CC^n,0) \to (\CC^N,0)$ 
  can only be nonzero in degrees $q$ with 
  \[
    q= k (n+1) - (k-1) N - s
  \]
  for $1<k\leq \left\lfloor \frac{N}{N-n} \right\rfloor$ and $0\leq s \leq \dim \Inst(f)$.
  In the particular case $N = n+1$ one has 
  \[
    n-\dim \Inst(f) \leq q \leq n
  \]
  as the range for possibly nonzero Betti numbers.
   \label{thm:MainTheorem}
\end{theorem}

Two more points in the statement of this theorem need clarification. The first one is 
the notion for a map germ $f$ to be \textit{dimensionally correct}. This means that the  multiple point spaces of $f$ have the expected dimension, as 
explained in Definition \ref{defDimensionallyCorrect}. The second one is the condition on 
$f$ to have corank $1$. The corank of a map germ centered at the origin is defined as 
\begin{equation}
  \corank f = n - \rank \D f(0), 
  \label{eqn:DefinitionCorank}
\end{equation}
where, as usual, $\D f$ denotes the differential of $f$. The dimensionally correct and corank $\leq 1$ conditions in the statement of Theorem \ref{thm:MainTheorem} ensure that the multiple point spaces are complete intersections of the expected dimensions, with singularities only over the instability locus; and that the multiple point spaces of a stable perturbation are smoothings of these complete intersections. This is explained in \cite{MararMond89}, which, together with \cite{GoryunovMond93}, forms the basis of this article. We close this section by giving some introductory examples:

\begin{example}\label{exCuspidalEdge}
The Cuspidal Edge 
\[S_\infty\colon (x,u)\mapsto (x^2,x^3,u),\]
admits deformations to the germs 
\[S_k\colon (x,u)\mapsto (x^2,x^3+u^{k+1}x,u), \quad k\geq 0.\]
and to their stable perturbations, see Figure \ref{fig:Example1}. 
For each $k$, these deformations may be obtained by taking suitable curves in the parameter space of the unfolding
 \[(x,u,a,b)\mapsto(x^2,x^3+x(a+bu^{k+1}),u,a,b).\]

The curve with parameter $t$ given by $t\mapsto(a,b)=(0,t)$ gives a family with generic fiber $S_k$ and special fiber $S_{\infty}$. Fixing $a=t^2$ and $b=t$ produces stable perturbations of $S_\infty$, with a different second Betti number for each $k$.
 The statement on the Betti numbers follows from the fact that the generic maps on these families are precisely the stable perturbations of $S_k$ (the ones obtained by fixing a small enough $b\neq 0$ and letting $a=t$). Since $S_k$ has $\mathcal A_e$-codimension $k$, and Mond's conjecture holds for germs $(\mathbb C^2,0)\to (\mathbb C^3,0)$, the second Betti number of the image of these stable perturbations is $k$.

The germ $S_\infty$ also admits deformations to the ``nodal edge''
\[(x,u)\mapsto (x^2,x^3+tx,u),\]
which can be retracted to a simple node curve. This gives a stable perturbation of a germ of a 
parametrized surface singularity having non trivial homology in dimension 1.

\end{example}
\begin{figure}[h]
  \centering
  \includegraphics[scale=1.05]{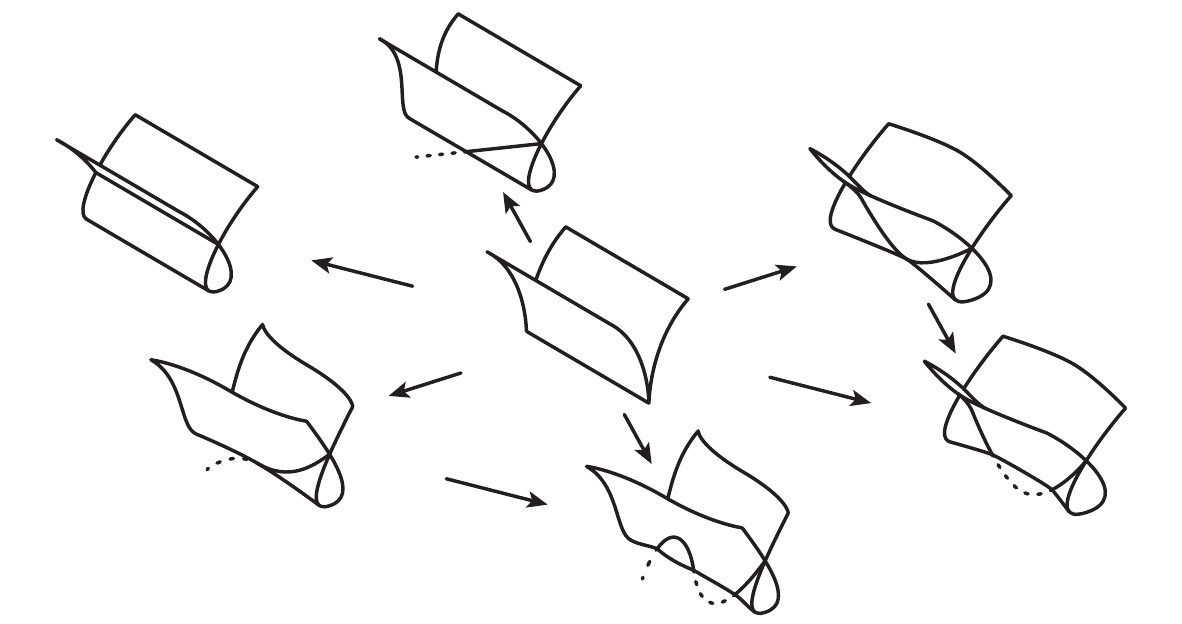}
  \caption{Unfoldings of the Cuspidal Edge}
  \label{fig:Example1}
\end{figure}

\begin{example}\label{exDeformationsBInfty}
The singularity 
\[B_\infty\colon (x,u)\mapsto(x^2,u^2x,u)\]
 admits deformations to the $\mathcal A$-finite singularities
\[B_k\colon (x,u)\mapsto(x^2,u^2x+x^{2k+1},u), \quad k\geq 1,\]
see Figure \ref{fig:Example2}.
Each $B_k$ has codimension $k$, and thus admits a deformation to a singularity whose image has the homotopy type of a wedge of $k$ spheres of dimension two.
We shall illustrate this by taking suitable curves in the parameter space of the deformation
\[(x,u)\mapsto(x^2,u^2x+x(a+b x^2+cx^4),u),\]
with coordinates $(a,b,c)\in \CC^3$.

The curve $t\mapsto (a,b,c) = t\cdot (1/10,-1,1)$ gives, for a small non zero parameter $t$, a stable perturbation
$f_t$, depicted at the left top corner of Figure \ref{fig:Example1},
which exhibits non-trivial homology in dimensions 1 and 2.

$t\mapsto (0,t,0)$ yields, for each $t\neq 0$, a singularity of type
$B_1$.  The curve $t\mapsto (-t,1,0)$ gives stable perturbations of $B_1$ which
exhibit two cross-caps and homology in dimension 2. The curve $t\mapsto
(-t,t,0)$ is similar, but it collapses to $B_\infty$.

 Similarly, the curve $t\mapsto (0,0,t)$ yields a singularity of type $B_2$ for each $t\neq 0$. The curve $t\mapsto (-t,0,1)$ consists of stable perturbations of $B_2$, with two cross-caps and rank of $H_2(Z_t)=2$ (One of the cycles cannot be seen on the real numbers.  You may think of this cycle, however, as obtained by glueing two disks by their boundaries, along the non-visible curves that continue the visible double-point curves after the cross-caps).

In all the processes above, one always chooses a suitable representative 
of the unfolding and a Milnor ball $B$, which in turn restricts the range 
of admissible parameters in the unfolding. 
This will be 
made precise in Section \ref{sec:MilnorFibrations}, but is already 
indicated in the illustrations in Figure \ref{fig:Example1}.
We would like to remark on the importance of the choices of 
neighbourhoods for the different disentanglements. 
Suppose the ball $B$ around the origin in $Z_0 \subset \CC^3$ has been 
chosen to allow the necessary range $t \in [0,1] \subset \RR$ in the deformations 
to $B_1$ and $B_2$ individually. Then, nevertheless, the curve 
\[
  [0,1] \to \CC^3, \quad
  t\mapsto (a,b,c) = (t(t-1),1-t,t)
\]
 \textit{cannot} be an admissible family with the given choice of $B$. 
Over $t \in (0,1)$ one has only stable fibers, which suggests 
that restricted to this range of parameters one has a trivial 
family. However, this is not the case: As we saw before, one has 
$h_2(Z_t\cap B) = 1$ for $t$ close to $0$, but 
$h_2(Z_t \cap B) = 2$ for $t$ close to $1$. 
In fact, one of the two cycles generating the second homology 
group grows as $t$ decreases from $1$ to $0$ and eventually 
leaves the chosen ball $B$. 
\begin{figure}[h]
  \centering \includegraphics[scale=1]{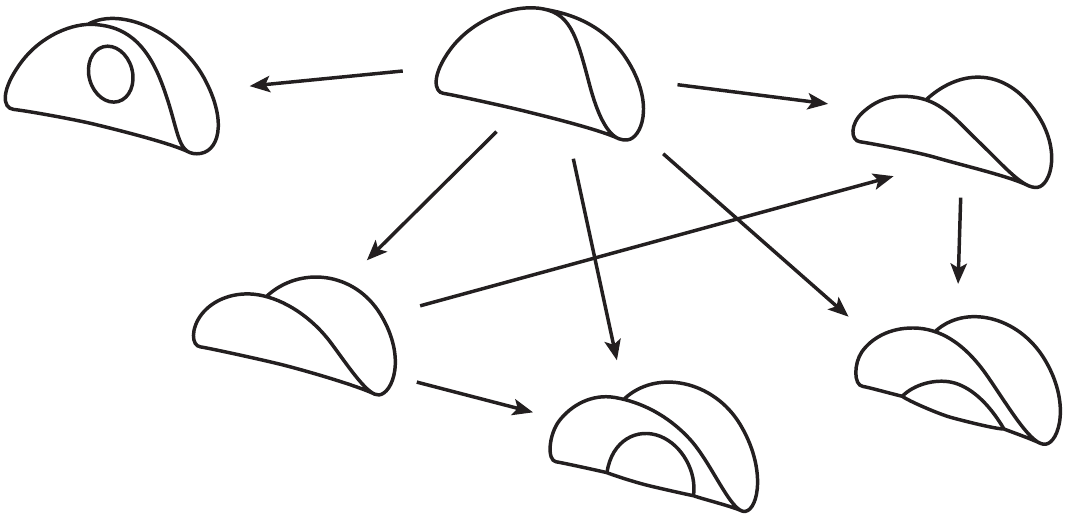}
  \caption{Some adjacencies around $B_\infty$} 
  \label{fig:Example2}
\end{figure}
\end{example}

The nodal edge as a stable perturbation of the cuspidal edge is a particular case of the following construction:
\begin{example}
One can produce map germs $f\colon (\CC^n,0)\to (\CC^{n+1},0)$ with  instability locus of dimension $d$ and a stable perturbation $f_s$ having non-trivial homology at dimension $n-d$, as follows: Choose any $\mathcal A$-finite unstable germ $g\colon (\CC^{n-d},0)\to(\CC^{n-d+1},0)$ of corank one and $g_s$ a stabilization of $g$. Now let $f$ and $f_s$ be $d$-parametric trivial unfoldings of $g$ and $g_s$. Since $g$ is $\mathcal A$-finite but not stable, it follows that the instability locus of $g$ is the origin. On the one hand, the space $T=\CC^d\times\{0\}$ corresponding to the trivial parameters of $f$ must be contained in the instability locus of $f$, because it is non trivially unfolded by $f_s$, just as $g$ is non trivially unfolded by $g_s$ at the origin. On the other hand, unfoldings of stable germs are stable, because their unfoldings are also unfoldings of the original germ and hence trivial. From this it follows that the instability locus of $f$ is $T$, and that $f_s$ is a stabilization of $f$.

Since $(g_s,s)$ is a stabilization of a finitely determined non stable map, the associated stable perturbation has the homotopy type of a bouquet of spheres of dimension $n-d$. Our claim follows because $(f_s,s)$ is a trivial unfolding of $(g_s,s)$, and therefore the image of $g_s$ is a deformation retract of the image of $f_s$.
\end{example}

\section{Preliminaries}
\subsection{$\mathcal A$-equivalence and unfoldings}
\label{sec:AEquivalenceAndUnfoldings}

In this work, maps and (multi-)germs of maps are considered
up to $\mathcal A$-equivalence, i.e. 
biholomorphisms in source and target. 
Standard references for this are \cite{TopStab76} and \cite{Golu-Gui}.

Throughout the text, a (multi-)germ is a germ of map the form $f\colon (\CC^n,S) \to (\CC^N,y)$, where $S$ is a finite set and $f(S)=\{y\}$. Two map germs 
\[
  f, g\colon (\CC^n,S) \to (\CC^N,0)
\]
are called $\mathcal A$-equivalent if there exist germs of biholomorphism 
\[
  \Phi \colon  (\CC^n,S) \to (\CC^n,S), \quad \Psi \colon  (\CC^N,0) \to (\CC^N,0),
\]
such that the following diagram commutes:
\[
  \xymatrix{
    (\CC^n,S) \ar[d]_\Phi \ar[r]^{f} & 
    (\CC^N,0) \ar[d]^{\Psi} \\ 
    (\CC^n,S) \ar[r]^g & 
    (\CC^N,0)
  }.
\]
Given such a germ $f \colon (\CC^n,0) \to (\CC^N,0)$ 
an \textit{unfolding} of $f$ on $r$ parameters is a map germ 
\[
  F \colon  (\CC^n \times \CC^r,(0,0)) \to ( \CC^N \times \CC^r,(0,0)), \quad
  (x,u) \mapsto (f_u(x), u)
\]
such that for $u = 0$ we obtain $f_0 = f$. 
Two unfoldings are called \textit{$\mathcal A$-equivalent as unfoldings}
if they are $\mathcal A$-equivalent via two germs of biholormorphism
of the form $\Phi=\phi\times \id_{\CC^r}$ and $\Psi=\psi\times
\id_{\CC^r}$.  An unfolding $F$ of $f$ is called trivial if it is
$\mathcal A$-equivalent, as an unfolding, to the trivial unfolding
$f\times \id_{\CC^r}$.

There are obvious analogous statements for finite maps 
$f \colon  X \to Y$ between complex manifolds.
Also, replacing biholomorphisms by homeomorphisms, we obtain the
definitions of topological $\mathcal A$-equivalence and topologically
trivial unfoldings.

\begin{definition}\label{defStability}
  A multi-germ $f\colon (X,S) \to (Y,0),$ 
  is stable if every unfolding $F$ of $f$ is trivial. A
  finite map $f\colon  X \to Y$ is stable at a point $y \in Y$ if the germ $f \colon 
  (X,f^{-1}(\{y\})) \to (Y,y)$ is stable. Finally, a finite map is
  \emph{stable} if it is stable at every point.  Topological stability is
  defined in the obvious analogous manner.
  
The \emph{instability locus} $\Inst(f)$ of a finite map is the subspace of points $y\in Y$ where the multi-germ of $f$ at $S=f^{-1}(\{y\})$ is not stable.  It is well known that $\Inst(f)$ is a closed complex subspace of $Y$.
  \end{definition}

\begin{definition} 
  \label{def:StabilizationAndStablePerturbation}
  A \emph{stabilization} of a germ $f$ is a one-parameter
  unfolding of $f$, having a representative 
  \[F=(f_u,u)\colon  X\times U  \to Y \times U\] so that 
  $f_s$ is stable for all $s\in U\setminus\{0\}$. Topological stabilizations are defined analogously.

  Let $Z=F(X\times U)$. By shrinking $X, Y$ and $U$, we may ask $F$ to satisfy the following conditions:
  \begin{enumerate}[i)]
    \item the family $\pi\colon  Z\to U$ is a locally trivial fibration 
      over $U\setminus\{0\}$,  
    \item the central fiber $\pi^{-1}(\{0\})$ is contractible,
    \item the space $Z$ retracts onto the central fiber $f_0(X)$.
  \end{enumerate}
  Such a mapping $F$ is called a \emph{good representative} of the stabilization, and each of the mappings $f_s\colon X\to Y$
  is called a \emph{stable perturbation} of $f$. 
\end{definition}

It is well known that for $(n,N)$ in the nice dimensions of Mather \cite{MatherVI}, with $n< N$,
all finite multi-germs admit a stabilization (observe that finite is equivalent to $\mathcal K$-finite whenever $n< N$). Away from the nice dimensions,
every multi-germ admits a topological stabilization.

\subsection{ Multiple point spaces }
Here we introduce the main objects needed for our result: the
multiple point spaces of a finite map. Our exposition is close to
\cite{NunoPenafort17} in spirit, but its contents
can be extracted from previous work like \cite{MararMond89}.

\begin{definition} A \emph{strict $k$-multiple point} of $f\colon X\to Y$ is a point
$(x^{(1)}, \dots, x^{(k)})\in X^k$ such that $f(x^{(i)})=f(x^{(j)})$
and $x^{(i)}\neq x^{(j)}$, for all $i\neq j$.

The $k$-th multiple point space of a stable map $f\colon
  X\to Y$ is the analytic closure 
  \[
    D^k(f)=\overline{\{\text{strict $k$-multiple points of $f$}\}}.
  \] 
  The space $D^k(f)$ is taken with the reduced analytic structure and
  the definition extends to stable multi-germs by taking representatives.

  The multiple point space of a finite multi-germ $f\colon (\CC^n,S)\to
  (\CC^N,0)$ is 
  \[
    D^k(f)=D^k(F)\cap (\CC^n\times\{0\})^k,
  \]
  for a stable
  unfolding $F\colon (\CC^n\times\CC^r,S\times\{0\})\to (\CC^N\times
  \CC^r,0)$ of $f$. The multiple point space $D^k(f)$ of a finite map
  $f\colon X\to Y$ is obtained by gluing the $k$-multiple point spaces
  of the multi-germs of $f$ at $f^{-1}(y)$, for all $y\in Y$. We write
  $D^k=D^k(f)$ whenever $f$ is understood.  
\end{definition} 

For every
$k$, the projection which drops the $i$-th coordinate takes strict
multiple points to strict multiple points. For any finite map $X\to
Y$, one can see that these projections induce maps 
\[
  \epsilon^{i,k} \colon D^k\to D^{k-1}.
\] 
For formal reasons, it is convenient to
extend this notation to the spaces $D^1=X$ and $D^0=Y$, to the maps
$\epsilon^{i,2}\colon D^2\to X$ dropping the $i$-th coordinate,
and to the map $\epsilon^{1,1}=f\colon X\to Y.$ We also write
\[
  \epsilon^k=\epsilon^{k,k}\circ\dots\circ\epsilon^{1,1}\colon D^k\to Y.
\]

The following well known result makes precise which diagonal
points are added when defining double points via passing to a stable
unfolding. A proof which is consistent with this approach can be found
in \cite{NunoPenafort17}.
\begin{lemma}\label{lemDoubleStrictAndSingular} 
  The double point space
  $D^2$ of any map $f\colon X\to Y$ consists of 
  \begin{enumerate} 
    \item
      Strict double points $(x,x')\in X^2$, with $x'\neq x$ and $f(x')=f(x)$,
    \item Diagonal singular points $(x,x)\in \Delta X$, such that $f$ is
      singular at $x$.  
  \end{enumerate}

\end{lemma}

The $k$-th multiple point space $D^k(F)$ of a finite map $F\colon
X\times U\to Y\times U$, of the form $F(x,u)=(f_u(x),u)$, can be
embedded in $X^k\times U$. This embedding makes $D^k(F)$ into a
(possibly non-flat) family of spaces $D^k(f_u)$, in the sense that
\[
  D^k(f_t):= D^k(F)\cap\{u=t\}
\]

In the case of corank
one germs, the multiple point spaces can be computed directly without
taking stable unfoldings. Every corank one map germ $(\CC^n,0)\to (\CC^p,0)$ can be taken, by a suitable change coordinates in the source and target, to a map of the form \[f(x,u)=(f_n(u,x),\dots,f_N(u,x),u),\quad u\in \CC^{n-1},x\in \CC.\] A set of generators for the ideal defining $D^k$ for such a map was given By Marar and Mond in \cite{MararMond89} (see Proposition 2.16). These generators are as follows:

\begin{lemma}
 In the setting above, the multiple point space $D^k(f)$
  is defined in $\CC^k\times\CC^{n-1}$ by the vanishing
  of the iterated divided differences 
  \[f_j[x_1,x_2,u], f_j[x_1,x_2,x_3,u],\dots,f_j[x_1,\dots,x_k,u],\text{ for }  j=n,\dots,N.\]
  The divided differences are defined as
  \begin{itemize} 
    \item[] $f_j[x_1,x_2,u]=\dfrac{f_j(x_2,u)-f_j(x_1,u)}{x_2-x_1}$, 
    \item[] $f_j[x_1,x_2,x_3,u]=\dfrac{f_j[x_1,x_3,u]-f_j[x_1,x_2,u]}{x_3-x_2}$
  \end{itemize} 
  and, iteratively, 
  \begin{itemize} 
    \item[]
      $f_j[x_1,\dots,x_k,u]=\dfrac{f_j[x_1,\dots,x_{k-2},x_k,u]-f_j[x_1,\dots,x_{k-2},x_{k-1},u]}{x_k-x_{k-1}},$
  \end{itemize} 
 \end{lemma}

Having explicit generators for monogerms, we can deduce the following:
\begin{lemma}
  Let $f\colon X\to Y$ be a finite map between holomorphic manifolds of 
  dimension $n$ and $N$ respectively. If $D^k$ is non-empty, then
  $\dim D^k\geq kn-(k-1)N$. If $f$ has corank one and $\dim D^k=kn-(k-1)N$,
  then $D^k$ is locally a complete intersection.
\end{lemma}

\begin{definition}\label{defDimensionallyCorrect}
  We say that the multiple point space $D^k$ of map 
  $f\colon X\to Y$ as above has the expected
  dimension if it is empty or has dimension $kn-(k-1)N$. A map is called
  \emph{dimensionally correct} if all its multiple point spaces have the
  expected dimension.
\end{definition}

Observe that, to check if a finite map is dimensionally correct, it
is enough to check dimensions from $D^2$ to the smallest $k$ such
that $D^k=\emptyset$. This is because $D^k=\emptyset$ forces the higher
multiplicity spaces to be empty, and hence to have the expected dimension.

We gather some results from \cite{GoryunovMond93} and \cite{MararMond89} 
describing the relation between stability and multiple point spaces 
in the corank one setting.

\begin{theorem}\label{thm:MararMondCurvilinearStabilityCriterion} 
  A corank one map $f$ is stable if and only if every $D^k$ is smooth of
  dimension $kn-(k-1)N$ or empty.  In this case, the following properties
  are satisfied: 
  \begin{enumerate} 
    \item Every $D^k$ is the closure of the
      set  of strict $k$-multiple points.  
    \item The maps $\epsilon^{i,k}\colon
      D^{k}\to D^{k-1}$ are corank one stable maps. 
    \item The multiple point
      spaces satisfy the iteration principle 
      \[
	D^k(\epsilon_r)\cong D^{r+k-1},
      \]
      via isomorphisms which commute with the maps $D^r(\epsilon_k)\to
      D^{r-1}(\epsilon_k)$ and $D^{r+k}\to D^{r+k-1}$.  
  \end{enumerate}
\end{theorem}

\begin{lemma} 
  For every finite map $f$, the maps $\epsilon^{i,k}\colon
  D^k\to D^{k-1}$ and $\epsilon^k\colon D^k\to Y$ are finite.  
\end{lemma}
\begin{proof}
  It suffices to show the claim for $\epsilon^{i,k}$, because
  every $\epsilon^k$ is a composition of those. Recall that the spaces $D^k$
  are obtained by glueing the corresponding spaces for representatives of
  multi-germs at $S=f^{-1}(\{y\})$, and thus it suffices to show the claim for
  multi-germs. Observe that the restriction of a finite map is finite and the
  maps $\epsilon^{i,k}$ are restrictions of those of any unfolding. Hence,
  since every finite multi-germ admits a stable unfolding, we may assume
  that $f$ is a stable finite multi-germ. In this case, the maps $D^{k}\to
  D^{k-1}$ are stable maps from a manifold of dimension $kn-(k-1)N$ to
  another of dimension $(k-1)n-(k-2)N$, hence finite.  
\end{proof}

\begin{corollary}
  \label{cor:CurvilinearMapsHaveCompleteIntersectionsAsMultPtSp}
  Let $X$ and $Y$ be complex manifolds of dimension $n$ and $N$ respectively.
  For any finite and dimensionally correct corank one map $f\colon X\to
  Y$, with instability locus of dimension $d$, the spaces $D^k$ are
  locally complete intersections of dimension $kn-(k-1)N$, if not empty,
  with singular locus of dimension at most $d$.  
\end{corollary}

\begin{corollary}
  \label{cor:EstimateOnSingularLocusForMultPtSpc}
  For any finite dimensionally correct corank one map $f\colon X\to
  Y$ as above, the union of the projections to $X$ of the singular loci of the $D^k$ is equal
  to the instability locus of $f$. In particular, if the instability locus
  has dimension $d$, then there is at
  least one space $D^k$ having a $d$-dimensional singular locus.  
\end{corollary}

\section{Milnor fibrations and disentanglements}
\label{sec:MilnorFibrations}

In this section we shall discuss two setups in parallel which may be seen as two 
branches of a common theoretical trunk: Smoothings of non-isolated singularities 
and disentanglements for finite map germs. Both setups fit into the  
commutative diagram (\ref{eqn:DeformationDiagram}) below: In the first case we 
consider $(Z_0,0) \subset (\CC^N,0)$ as a non-isolated singularity and 
$(Z_0,0) \hookrightarrow (Z,0) \overset{h}{\longrightarrow} (\CC,0)$ a smoothing of it 
over $(\CC,0)$ with total space $(Z,0)$. In the other case, $(Z_0,0)$ appears as the 
image of a finite holomorphic map germ $f \colon (\CC^n,0) \to (\CC^N,0)$. Passing 
to an unfolding $F \colon (\CC^n,0) \times (\CC,0) \to (\CC^N,0) \times (\CC,0)$, we can 
easily verify that also $F$ is finite and hence has an analytically closed image 
$(Z,0) \subset (\CC^{N+1},0)$ equipped with a projection $h$ to the base of the 
unfolding. Forgetting about $F$ we obtain a deformation of $(Z_0,0)$ over $(\CC,0)$ 
in the classical sense.

\begin{equation}
  \xymatrix{
    &
    (\CC^N,0) \ar@{^{(}->}[rr] \ar[dd] & 
    &
    (\CC^{N+1},0) \ar[dd]^u \\
    (Z_0,0) \ar@{^{(}->}[rr] \ar@{^{(}->}[ur] \ar[dr] & 
    & 
    (Z,0) \ar@{^{(}->}[ur] \ar[dr]^h \\
    & 
    \{0\} \ar@{^{(}->}[rr] & 
    &
    (\CC,0)
  }
  \label{eqn:DeformationDiagram}
\end{equation}
We will usually write $y$ for the standard coordinates of $(\CC^N,0)$ and 
$(y,u)$ for those of $(\CC^{N+1},0)$ with $u$ the deformation parameter. 
The function $h$ is the restriction of $u$ to the total space 
$(Z,0)$ of the deformation. 

\subsection{The fibration theorems}
\label{sec:FibrationTheorems}

Next, we will make precise what we understand to be the Milnor fiber in this 
situation closely following \cite{Le77}. 
By abuse of notation we will denote by $Z_0$, $Z$, and $h$ 
a set of representatives of the corresponding germs defined on some product
$Y \times \Delta \subset \CC^{N+1} = \CC^N \times \CC$ with $Y$ an open neighborhood 
of $0$ in $\CC^N$ and $\Delta$ an open disc around the origin in the deformation base 
$\CC$. 

One starts out by endowing $(Z,0)$ with a Whitney stratification 
$\Sigma = \{ S_\alpha\}_{\alpha \in A}$, in which 
$(Z_0,0)$ appears as union of strata. 
For the existence of complex analytic Whitney stratifications 
see e.g. \cite{Whitney65}. 
If (\ref{eqn:DeformationDiagram}) is a 
smoothing of $(Z_0,0)$, 
then $Z \setminus Z_0$ is itself necessarily smooth in a neighborhood
of the origin and in this case we have a unique open stratum which we will call 
$S_0 := Z\setminus Z_0$. Whenever (\ref{eqn:DeformationDiagram}) arises from an unfolding 
$F : (\CC^n \times \CC, 0) \to (\CC^N\times \CC,0)$ of a finite holomorphic 
map germ $f = f_0 : (\CC^n,0) \to (\CC^N,0)$, the open stratum $S_0$ might 
not coincide with $Z \setminus Z_0$, since the deformation 
of the image $(Z_0,0)$ induced 
from the unfolding $F$ does not have to be a smoothing of $(Z_0,0)$.
In either case we may utilize the Curve Selection Lemma to verify 
that the restriction of $h$ to any 
stratum $S_\alpha$ outside $Z_0$
does not have a critical point in a neighborhood of $0\in Z$. 
After shrinking our representatives, we may therefore assume that 
$h|_{S_\alpha}$ is a submersion at any point $p\in S_\alpha$, $p \notin Z_0$ and write 
\[
  T_{p} h^{-1}(\{h(p)\}) = \ker \D h|_{S_\alpha} \subset T_{p} S_\alpha
\]
for the tangent space of its fiber $h^{-1}(\{h(p)\})$ through $p$.
Note that by construction the restriction of $h$ to any stratum $S_\alpha \subset Z_0$ 
in the central fiber is the constant map. Therefore at any point
$p \in S_\alpha \subset Z_0$ 
the spaces $T_p S_\alpha$ and 
$T_p h^{-1}(\{h(p)\})$ coincide. 

As a second step one refines the stratification in order to also fulfill 
Thom's $a_h$-condition. This is the crucial ingredient in order to establish 
the fibration theorems, see \cite{Thom69}.
The existence of such refinements 
in the complex analytic setting has been shown in 
\cite{Hironaka77}. 

\begin{definition}
  Let $h \colon (Z,p) \to (\CC,q)$ be a holomorphic map germ 
  and $\Sigma = \{S_\alpha\}_{\alpha \in A}$ a Whitney stratification of $(Z,p)$
  as above.
  Consider a sequence of points $(p_i)_{i\in \NN}$ in a stratum $S_\alpha$ 
  converging to $p \in S_\beta$ such that the sequence of tangent spaces 
  \[
    T_{p_i} S_\alpha \overset{n\to \infty }{\longrightarrow} 
    T_{\infty} \supset T_{p} S_{\beta} 
  \]
  also converges. From this one also 
  obtains the sequence of \textit{relative} tangent spaces 
  \[
    K_i := T_{p_i}h^{-1}(\{h(p_i)\}) \subset T_{p_i} S_\alpha.
  \]
  Passing to a subsequence, we may assume that also the $K_i$ converge 
  to a limit $K_{\infty} \subset T_\infty$.
  We say that the stratification $\Sigma$ satisfies the $a_h$-condition 
  at $p$, if for \textit{any} such sequence one has
  \begin{equation}
    K_\infty \supset T_{p} h^{-1}\left( \{h(p)\} \right)
    \label{eqn:a_h-condition}
  \end{equation}
  where $T_{p}h^{-1}(\{h(p)\}) = \ker \D h|_{S_\beta}(p) \subset T_p S_\beta$.
\end{definition}

In a third step we restrict to \textit{proper} representatives by choosing a 
\textit{Milnor ball}.  Suppose 
\[
  F = (f_u,u): X \times \Delta \to \CC^N \times \Delta
\]
is a representative of the unfolding $F$ as above defined on some open neighborhood 
$X \subset \CC^n$ over a small disc $\Delta \subset \CC$. Let $F(X)$ be its image 
in $\CC^N$. We may choose the open neighborhood $0 \in Y \subset \CC^N$ such that 
$Z := F(X) \cap Y \times \Delta$ is analytically closed in $Y\times \Delta$. 
A Milnor ball is now chosen 
as follows: Consider the squared distance function from the $u$-axis in $\CC^{N+1}$
\[
  \rho : (y,u) \mapsto |y|^2.
\]
This is a real analytic function. We can apply the Curve Selection Lemma to deduce 
that $\rho$ has no critical values on some small interval $(0,\varepsilon_0]$ with 
respect to the stratification $\Sigma$ in a neighborhood of the origin.
We may assume this neighborhood to be $Y\times \Delta$ again.
Consequently,
\[
  \overline Z := Z \cap \rho^{-1}( [0,\varepsilon_0] ) \quad \textnormal{ and } \quad
  \overline Z_0 := Z_0 \cap \rho^{-1}( [0,\varepsilon_0 ])
\]
are again Whitney stratified spaces. Let 
\[
  \partial \overline Z := Z \cap \rho^{-1}(\{\varepsilon_0\}) 
  \quad \textnormal{ and } \quad
  \partial \overline Z_0 := Z_0 \cap \rho^{-1}(\{\varepsilon_0\}) 
\]
be their boundaries. They inherit a Whitney stratification 
$\partial \Sigma$ from $\Sigma$ by intersecting each stratum 
$S_\alpha$ with $\rho^{-1}(\{\varepsilon_0\})$.
Moreover, the restriction 
\[
  h|_{\overline Z} : \overline Z \to \Delta
\]
is proper and, as a consequence of Thom's first Isotopy Lemma applied to $\rho$ on 
$Z_0$, one has:

\begin{corollary}[Conical structure]
  \label{cor:ConicalStructure}
  The inclusion $\{0\} \hookrightarrow \overline Z_0$ is a deformation retract.
\end{corollary}

\noindent

In this setup, L\^e provided the following 
fibration theorem, cf. \cite{Le77}, Theorem 1.1 and Remark 1.3 (b): 

\begin{theorem}
  There exists $\varepsilon_0 \gg \eta >0$ such that 
  the restriction 
  \[
    h : \overline Z \cap h^{-1}(\Delta_\eta \setminus \{0\} ) \to \Delta_{\eta} \setminus \{0\}
  \]
  is a topological fibration, where $\Delta_\eta$ is the disc around the origin 
  of radius $\eta$.
  \label{thm:FibrationTheorem}
\end{theorem}

\begin{definition}
  In a given deformation of a space $(Z_0,0)$ as in 
  (\ref{eqn:DeformationDiagram}) and Theorem \ref{thm:FibrationTheorem}, we call 
  \[ 
    h^{-1}(\{t\}) \cap \overline Z, \quad t \in \Delta_\eta \setminus \{0\}
  \]
  the Milnor fiber of $h$ on $(Z,0)$. If (\ref{eqn:DeformationDiagram}) arises 
  from a stabilization $F$ of a holomorphic map germ $f$, then we will speak of 
  the disentanglement of $f(X)$ in the given unfolding $F$.
\end{definition}

  The reader should verify that with the particular choices for our representatives 
  made in this section, the requirements on a good representative in 
  Definition \ref{def:StabilizationAndStablePerturbation} are indeed satisfied. 

\begin{remark}
  The notions of this section are mostly of topological nature. A priori, the 
  space germs $(Z,0)$, and $(Z_0,0)$ are germs of complex spaces, i.e. they 
  might have a non-reduced scheme structure. However, when considering their 
  Whitney stratifications, we only consider their reduced models and the 
  reduced structure of the strata and consequently everything that follows 
  is to be understood purely topologically. 
\end{remark}

\subsection{Induced deformations of multiple point spaces}

\label{sec:InducedDeformationsOfMultiplePointSpaces}

In the setting we will encounter when dealing with the multiple point spaces,
we do not have the freedom to choose the shape of $\overline Z$ inside $Z$ since 
multiple point spaces and their deformations appear as preimages of the canonical 
projections to the unfolding of the original map germ, which had been chosen 
beforehand: Suppose as usual that we are given a representative 
\[
  F = (f_u,u) : X \times \Delta \to Y \times \Delta \subset \CC^N \times \CC
\]
of an unfolding of a finite holomorphic map germ $f = f_0$. Let 
$Z \subset Y \times \Delta$ be its image and $D^k$ the 
multiple point spaces: 
\medskip
\begin{equation}
  \xymatrix{ 
    Z \ar[d]_h & 
    D^1 \ar[l]|-F \ar[dl]|-u &
    D^2 \ar[l]^{\varepsilon^{\bullet,2}} \ar@/_1pc/[ll]|-{\varepsilon^2} \ar@/^/[dll]|-u&
    D^3 \ar[l]^{\varepsilon^{\bullet,3}} \ar@/_2pc/[lll]|-{\varepsilon^3} \ar@/^/[dlll]|-u&
    \cdots \ar[l] \ar@/_3pc/[llll] \ar@/^/[dllll] \\
    \Delta & & & & \\
  }
  \label{eqn:MultiplePointSpacesUnfolding}
\end{equation}
Consider the squared distance $\rho$ from the $u$-axis in $\CC^{N+1}$ as before. 
On all the spaces $D^k$ 
this already fixes the function with respect to which our proper representatives 
are about to be chosen: 
\begin{equation}
   \rho_k := \rho \circ \varepsilon^k : D^k \to \RR
  \label{eqn:DistanceFunctionMultiplePointSpaces}
\end{equation}
From the finiteness and analyticity 
of all the maps $\varepsilon^k$ we easily deduce the following three 
properties of all the $ \rho_k$:

\begin{enumerate}
  \item $ \rho_k$ has an isolated zero on the central fibers 
    $D^k_0 = D^k \cap \{u=0\}$, resp. $Z\times \{0\}$, at the origin.
  \item There exists $\varepsilon_0>0$ such that for each 
    $0< \varepsilon \leq \varepsilon_0$ and all $k>0$ the restriction 
    \[
      u : \overline D^k = D^{k} \cap  \rho_k^{-1}( [0,\varepsilon_0] ) \to \Delta
    \]
    is proper.
  \item Furthermore, $\varepsilon_0 >0$ can be chosen small enough such that each 
    $ \rho_k$ is a stratified submersion on
    $D_0^k \cap \rho_k^{-1}( (0,\varepsilon_0])$.
\end{enumerate}
\noindent
We will from now on assume that $\varepsilon_0$ has been chosen small enough 
to fulfill all these requirements simultaneously with those formulated in 
the previous section.

\medskip 

Recall Theorem \ref{thm:MararMondCurvilinearStabilityCriterion}, 
Corollary \ref{cor:CurvilinearMapsHaveCompleteIntersectionsAsMultPtSp}, and 
Corollary \ref{cor:EstimateOnSingularLocusForMultPtSpc}: If 
(\ref{eqn:MultiplePointSpacesUnfolding}) is a stabilization of a dimensionally correct 
map germ $f_0$ of corank one with instability locus of dimension $d = \dim \Inst(f_0)$, then 
the induced deformations 
\[
  \overline D^k \overset{u}{\longrightarrow} \Delta
\]
are smoothings of the complete intersection singularities
\[
  (D^k_0,0) = (D^k \cap \{u=0\},0),
\]
which have non-isolated singular locus of dimension 
at most $d$. The only difference to the classical case is the unconventional 
choice of the Milnor balls via the functions $\rho_k$. But as we shall see below, 
this does 
not alter the theory and we can apply the usual machinery to get to know their topological 
invariants. The induced deformation of one of the multiple point schemes is what the 
reader should have in mind when progressing to the next section.

\subsection{Non-isolated complete intersection singularities}
\label{sec:Non-IsoCompleteInter}

Suppose we are given a smoothing of a nonisolated singularity $(Z_0,0) \subset (\CC^N,0)$ 
as in (\ref{eqn:DeformationDiagram}). 
There are essentially two different ways to construct transversal, isolated singularities 
from nonisolated ones. In the first case one considers the singular part of the 
\textit{link} $\partial \overline Z_0 = \rho^{-1}(\{\varepsilon_0\}) \cap Z_0$ as 
e.g. in \cite{Siersma83} and also the original work by Kato and Matsumoto 
\cite{KatoMatsumoto75}. In the other case, which is for example 
exhibited in \cite{Le77}, one takes 
hyperplane slices at the central point $0 \in Z_0$, which are transverse to the 
singular locus. A detailed description of the sense of transversality underlying this 
choice is outlined in \cite[Section 2]{HammLe73}.
Unfortunately, the article only deals with 
hypersurfaces, but the methods applied there carry over naturally:
We are considering functions $h \colon (Z,0) \to (\CC,0)$ on an arbitrary 
reduced, complex analytic space $(Z,0) \subset (\CC^{N+1},0)$, 
rather than functions on smooth spaces. All the choices of stratifications 
and their refinements in \cite{HammLe73} can easily be adapted 
to the situation in Diagram (\ref{eqn:DeformationDiagram}) and one eventually 
obtains the following theorem by L\^e, \cite{Le77}, Theorem 2.1:

\begin{theorem}
  There is an open Zariski dense set $\Omega$ in the space of affine 
  hyperplanes of $\CC^N$ at $0$ such that if $l=0$ belongs to $\Omega$,
  there exists an analytic curve $C_0 \subset W \subset \CC^2$ in 
  an open neighborhood $W$ of $0 \in \CC^2$, such that for any $\varepsilon >0$ 
  small enough and $\eta>0$, $\varepsilon \gg \eta$, the mapping 
  \[
    (h,l) : Z \cap B_\varepsilon \cap (h,l)^{-1}( B_\eta \setminus C_0 ) 
    \to B_\eta \setminus C_0
  \]
  is a topological fibration.
  \label{thm:LeHyperplaneFibration}
\end{theorem}

\noindent
As usual, $B_\varepsilon$ and $B_\eta$ denote the balls of radius 
$\varepsilon$ and $\eta$ in $\CC^N$ and $\CC^2$ respectively.
It is evident from the proof that 
we can replace $B_\varepsilon$ by $ \rho_k^{-1}( [0,\varepsilon] )$ for a suitable 
real analytic function $ \rho_k$, which fulfills the requirements 
formulated in Section \ref{sec:InducedDeformationsOfMultiplePointSpaces}. 
The verification of this fact along the lines of \cite{Le77} is left to the reader. 
See also \cite{HammLe73}.

The curve $C_0$ is the image of the 
so-called \textit{Polar curve} of the linear form $l$ relative to $h$ on $Z$: 
\begin{equation}
  \Gamma = \overline{ \{ y \in Z\setminus Z_0 : \ker \D h(y) \subset \{l=0\} \} }.
  \label{eqn:PolarCurveDefinition}
\end{equation}
For a fixed value $t \in \CC$, $|t|$ small enough, the restriction of $l$ to the 
interior of $\overline Z_t = \rho^{-1}([0,\varepsilon]) \cap h^{-1}(\{t\})$ has 
only isolated critical points -- the intersection points of $\Gamma$ with 
$h^{-1}(\{t\})$. One can furthermore show (see \cite{HammLe73})
that for a generic choice of $l$ 
these critical points are Morse critical points. But we shall not need that 
fact. 

What is more important for our purposes, is the fact that the set $\Omega$ of 
admissible hyperplanes 
$L = \{l=0\}$ is Zariski open. 
Let $\Omega'$ be the set of hyperplanes, which are transversal to the stratification 
$\Sigma$ of the total space $Z$ at $0$. By this we mean that the restriction 
of the linear form $l$ to any stratum of positive dimension is a submersion and 
that at an arbitrary point $p \in Z$ it does not annihilate any 
limiting tangent space of adjacent strata. 
Since the intersection of two Zariski open sets is again Zariski open and in particular 
non-empty, we may choose $L\in \Omega \cap \Omega'$.
With this choice it is easy to see that, 
since $Z_0$ is a union of strata, the \textit{transversal singularity} given by 
\begin{equation}
  Z_0^{\pitchfork} := Z_0 \cap L
  \label{eqn:TransversalSingularityDefinition}
\end{equation}
is nonsingular at every nonsingular point of $Z_0$. Moreover, 
$l$ is a non-constant function on every irreducible component of any 
stratum $S_\beta \subset \Sing(Z_0)$ 
of positive dimension and hence $\dim \Sing(Z_0^\pitchfork) = \dim \Sing(Z_0) -1$,
whenever $Z_0$ has non-isolated singularities.
This yields an induction argument on $d = \dim \Sing(Z_0)$: 

\begin{theorem}
  \label{thm:InductionStepConnectivity}
  Suppose $h : (Z,0) \to (\CC,0)$ as in (\ref{eqn:DeformationDiagram}) is a smoothing 
  of an $n$-dimensional,
  non-isolated singularity $(Z_0,0) \subset (\CC^N,0)$ 
  with singular locus $\Sing(Z_0)$ of dimension $d$ at $0$ and $L= \{l=0\}$ is 
  a hyperplane meeting the requirements of Theorem \ref{thm:FibrationTheorem}
  with transversal singularity 
  $(Z_0^\pitchfork,0)$. 
  Let $\rho$ be a non-negative real analytic function on a neighborhood of the origin 
  in $\CC^N$, which defines a good proper representative 
  $h : \overline Z \to \Delta$ of the deformation of 
  $\overline Z_0 = h^{-1}(\{0\})\cap \rho^{-1}([0,\varepsilon])$ and 
  the associated transversal singularity $\overline Z_0^{\pitchfork}$. Then 
  up to homotopy we obtain the space $\overline Z_t$ from $\overline Z_t^\pitchfork$ 
  by attaching handles of real dimension $n$. 
  \end{theorem}

The idea behind the proof of this theorem is to use the squared distance 
function $\lambda := |l|^2$ from the hyperplane $L$ as a real Morse function 
on $\overline Z_t$. The latter is a manifold with boundary and we therefore have to 
pass to stratified Morse theory. 
Perturbing the equation $l$ slightly, we may assume that 
the complex function $l$ has only Morse critical points on the interior 
of $\overline Z_t \setminus \overline Z_0^\pitchfork$ and it is easily seen that outside 
$L \cap \overline Z_t$, the function $\lambda$ has real Morse singularities 
of index $n = \dim_{\CC} \overline Z_t$ precisely at the critical points 
of $l$. 
The difficulty in the proof does therefore not arise from the critical points of 
$\lambda$ in the interior of $\overline Z_t\setminus \overline Z_0^\pitchfork$, but from critical points 
of $\lambda$ on the boundary $\partial\overline Z_t$. More precisely: 
The hardest part of the proof is to show that these critical points do not contribute 
to changes in the homotopy type when passing through their associated critical 
values.

In order to understand why this is the case, recall the following 
notions of Morse theory on manifolds with boundary.
Let $f \colon M' \to \RR$ be a smooth function on a manifold $M'$ and let 
$M \subset \overline M \subset M'$
be an open submanifold with boundary $\partial M$. For a point
$p\in \partial M$ choose a smooth function
$t \colon (M',p) \to (\RR,0)$ defining the boundary, i.e.
\[
  (\partial \overline M,p) = (t^{-1}(\{0\}),p) \quad \textnormal{and} \quad
  (\overline M,p) = (t^{-1}( (-\infty,0] ),p) .
\]
\begin{definition}
  We say that $p \in \partial M$ is a \textit{boundary critical point}
  for $f$ on $\overline M$, 
  if $p$ is a regular point of $f$ on $M'$ and a critical point of the 
  restriction $f|_{\partial M}$, i.e. if in the tangent space $T_p M'$ the 
  equation
  \begin{equation}
    \D f(p) = a \cdot \D t(p)
    \label{eqn:LinearDependenceBoundaryCriticalPoint}
  \end{equation}
  holds for some $a \neq 0$.
  The point $p$ is an \textit{outward} (resp. \textit{inward})
  critical point of $f$ on $\overline M$ 
  if the coefficient $a$ in (\ref{eqn:LinearDependenceBoundaryCriticalPoint}) is 
  positive (resp. negative). 
  Moreover, a boundary critical point $p$ is a \textit{Morse point} if 
  the restriction $f|_{\partial \overline M}$ has a Morse critical point 
  at $p$. 
\end{definition}

The following lemma can easily be deduced from stratified Morse theory, 
see \cite{GoreskyMacPherson88}.
\begin{lemma}
  Let $f \colon M' \to \RR$ be a smooth function on a manifold $M'$ 
  and $M \subset \overline M \subset M'$ a submanifold with boundary 
  $\partial \overline M$ such that the restriction of $f$ to $\overline M$ 
  is proper.
  Suppose $p \in \partial \overline M$ is 
  a boundary Morse critical point with critical value 
  $b \in \RR$ and no further critical points in the fiber 
  $\overline M \cap f^{-1}(\{b\})$.

  If $p$ is an outward critical point of $f$, then there exists 
  an $\varepsilon>0$ such that the space  
  $f^{-1}( (-\infty, b-\varepsilon] )$ is a strong deformation retract of 
  $f^{-1}( (-\infty,b+\varepsilon] )$. 
  
  If $p$ is an inward critical point of $f$, then up to homotopy we obtain 
  $f^{-1}( (-\infty, b+\varepsilon] )$ 
  from 
  $f^{-1}( (-\infty,b-\varepsilon] )$ by attaching a cell of dimension $i$, 
  where $i$ is the Morse index of the restriction $f|_{\partial \overline M}$.

  \label{lem:RelativeMorseTheory}
\end{lemma}

\begin{example}
  Consider ``Kenny's head'' (Figure \ref{fig:Kenny'sHead}), the manifold $\overline M$ given by the 
  intersection of the sphere $S^2 \subset \RR^3$ 
  with the half space $\{y \leq 1/ \sqrt{2} \}$. As usual we take
  the height function $z$ as a Morse function on $M$. There are 
  four critical points of $z$ on $\overline M$: Two on the interior 
  of $\overline M$ inside Kenny's throat and on the top of his head 
  and two boundary critical points close to his chin and on his forehead.
  Consider the spaces
  \[
    \overline M_c :=z^{-1}( (-\infty, c]) \cap \overline M
  \]
  for varying $c$. 
  It is easily seen that passing through 
  the critical value $-1/\sqrt{2}$ of the boundary critical point at 
  Kenny's chin does not change the homotopy type of 
  $\overline M_c$. Indeed, this is an outward critical point. 
  
  On the other hand, when we approach $c = 1/\sqrt{2}$, the critical 
  value coming from the critical point on the forehead, we attach a $1$-cell to 
  $\overline M_{c-\varepsilon}$. This is the case, because 
  the critical point on the forehead is inward pointing and the Morse 
  index of the restriction of $z$ to $\partial \overline M$ is 
  $1$.
  \begin{figure}[h]
  \centering
  \includegraphics[scale=1]{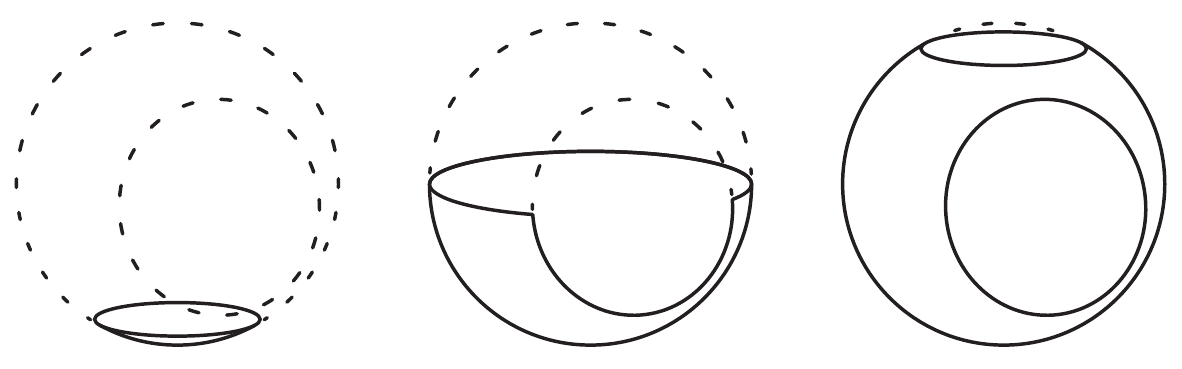}
  \caption{Slicing Kenny's head.}
  \label{fig:Kenny'sHead}
\end{figure}

\end{example}

\begin{lemma}
  In the setting of Theorem \ref{thm:InductionStepConnectivity}, let 
  $\lambda = |l|^2$ be the squared distance from $L$. 
  We say that $p \in \partial \overline Z \cap h^{-1}(\Delta_\delta\setminus\{0\})$  
  is a relative boundary critical point of $\lambda$, if it is 
  a boundary critical point on the fiber 
  \[
    \overline Z_t = h^{-1}(\{h(p)\}) \cap \rho^{-1}( [0,\varepsilon] )
  \]
  of $h$ over $t = h(p)$.
  We can choose an $\varepsilon >0$ small enough such that 
  \begin{enumerate}
    \item[i)] there are no relative boundary critical points in a 
      neighborhood of $\partial \overline Z_0^{\pitchfork}$;
    \item[ii)] there exists a $\delta >0$ such that 
      whenever $p \in \partial \overline Z$ is a relative boundary 
      critical point, then it is an outward critical point of 
      $\overline Z_{h(p)}$.
  \end{enumerate}
  \label{lem:OnlyOutwardCriticalPoints}
\end{lemma}

From the proof of this lemma we first extract a further technicality 
as a statement of its own.

\begin{lemma}
  Let $V$ be a finite dimensional vector space over a field $\mathbb K$ with 
  $\mathbb K = \RR$ or $\CC$. Let $(W_n)_{n\in \NN}$ be a sequence of 
  subspaces in $V$ converging to a limit space $W$ and let 
  $(\Phi_n)_{n\in \NN}$ and $(\Psi_n)_{n\in \NN}$ be linear forms on 
  $V$ converging to linear forms $\Phi$ and $\Psi$ respectively with 
  $\Psi \neq 0$. Suppose we have for each $n\in \NN$
  \[
    \Phi_n |_{W_n} = a_n \cdot \Psi_n |_{W_n}
  \]
  for some sequence of numbers $(a_n)_{n\in \NN}$ in $\mathbb K$. 
  Then $(a_n)_{n\in \NN}$ converges to some $a \in \mathbb K$ and 
  \[
    \Phi|_W = a \cdot \Psi|_W.
  \]
  \label{lem:technicalSublemma}
\end{lemma}

\begin{proof}
  Let $v \in W$ be a vector with $\Psi(v) \neq 0$ and choose a 
  sequence of vectors $(v_n)_{n\in \NN}$ with $v_n \in W_n$ and 
  $\lim_{n\to \infty} v_n = v$. Then due to continuity 
  \[
    a_n \cdot \Psi_n(v_n) = \Phi_n(v_n) \overset{n\to \infty}{\longrightarrow} 
    \Phi(v) = a \cdot \Psi(v)
  \]
  for some $a \in \mathbb K$. We immediately deduce that for $n$ big enough
  \[
    a_n = \frac{\Phi_n(v_n)}{\Psi_n(v_n)} 
    \overset{n\to\infty}{\longrightarrow}
    \frac{\Phi(v)}{\Psi(v)} = a,
  \]
  so in particular $(a_n)_{n\in \NN}$ converges. It follows a posteriori 
  that therefore also 
  \[
    \Phi|_W = a\cdot \Psi|_W
  \]
  as linear forms on $V$.
\end{proof}

\begin{proof}(of Lemma \ref{lem:OnlyOutwardCriticalPoints})
  We may choose the representative $Z$ so that we can 
  assume that $h$ has no critical points on $Z$ outside $Z_0$ and 
  also $\rho$ does not have critical points on $Z$ outside $\{ \rho = 0 \}$.
  Since $Z_0$ is a union of strata and the hyperplane $L$ is chosen transversal 
  to the stratification, this holds in particular for $Z_0$ and $Z_0^\pitchfork$. 
  For a point $p \in S_\alpha \subset Z$ let 
  $Z_{h(p)} = h^{-1}(\{h(p)\})$ be the fiber of $h$ through this point. 
  We denote its tangent space by
  \[
    K(p) = \ker \D h|_{S_\alpha}(p) = T_p h^{-1}(\{h(p)\}) = T_p Z_{h(p)}.
  \]
  In particular $K(p) = T_p S_\alpha$, whenever $p \in Z_0$, since $h$ 
  is constant on $Z_0$. The $a_h$-condition moreover assures that 
  for any sequence of points $(p_n)_{n\in \NN} \subset S_\alpha$ 
  converging some $p \in S_\beta$ one has 
  \[
    \lim_{n\to \infty} K(p_n) \supset K(p)
  \]
  whenever the limit exists.
  
  We first show $i)$. Certainly $\D \lambda(p) = 0$ at any $p \in L$. 
  However, by the choice of $L$ and $\varepsilon$ with respect to $Z$ we 
  have that for any $p \in \partial \overline Z_0^\pitchfork$, 
  $p \in S_\beta \subset Z_0$, the real differentials 
  \[
    \Re \D l(p), \quad \Im \D l(p), \quad \textnormal{and} \quad \D \rho(p)
  \]
  are linearly independent over $\RR$ on $T_p S_\beta$. 
  This is equivalent to saying that 
  \[
    \D l(p) \quad \textnormal{ and } \quad \D^{1,0} \rho(p) 
  \]
  are linearly independent over $\CC$, where $\D^{1,0} \rho(p)$ is 
  the holomorphic part of the differential of $\rho$ considered as 
  a map to $\CC$. Now suppose there was a sequence 
  $(p_n)_{n\in \NN}$ of points in some adjacent stratum $S_\alpha$ 
  converging to $p$ and such that at each $p_i$ one has 
  \begin{eqnarray*}
    \D l(p_n)|_{K(p_n)} &=&  c_n \cdot \D^{1,0} \rho(p_n)|_{K(p_n)}
  \end{eqnarray*}
  for some sequence of complex numbers $c_n$.
  Passing to a subsequence, if necessary, we may assume that the limit 
  $\lim_{n\to \infty} K(p_n) = K$ exists. Applying Lemma 
  \ref{lem:technicalSublemma} we obtain a contradiction to the linear 
  independence of $\D l(p)$ and $\D^{1,0} \rho(p)$ on $T_p S_\beta \subset K$.
  We infer that there exists an open neighborhood $U$ of 
  $\partial \overline Z_0^\pitchfork$ in $Z$ such that for each stratum 
  $S_\alpha$ and every point $q \in U \cap S_\alpha$ the differentials 
  $\D l(q)$ and $\D^{1,0} \rho(q)$ are linearly independent over $\CC$. 
  
  It is now easy to see from the formula 
  \[
    \D \lambda = 2 \Re (\overline l \cdot \D l)
  \]
  that on this open neighborhood also $\D \lambda$ and $\D \rho$ must 
  be linearly independent, whenever $\overline l \neq 0$, i.e. outside 
  $L$. The points on $L \cap \overline Z \cap U$ on the other hand 
  cannot be relative boundary critical points, since they are critical 
  points of $\lambda$ on the fiber $Z_t$ already. 

  Concerning $ii)$: It has been shown in \cite{Zach17}, Lemma 2.3.8,
  that we can 
  choose $\varepsilon>0$ small enough such that on $\overline Z_0$ 
  the following holds: Whenever in a point 
  $p \in \partial \overline Z_0$, $p \in S_\beta$, we have a linear 
  dependence 
  \begin{equation}
    \D \lambda(p)|_{T_p S_\beta} = a \cdot \D \rho(p)|_{T_p S_\beta},
    \label{eqn:LinearDependenceEquation}
  \end{equation}
  then $a \geq 0$ with equality if and only if 
  $p \in \partial \overline Z_0^\pitchfork$.

  We briefly recall the argument. It can be shown that 
  \[
    C^- = \overline{\{ p \in \partial \overline Z_0: 
      \D \lambda (p)|_{T_p S_\beta} = a \cdot \D \rho(p)|_{T_p S_\beta}, a <0 \}}
  \]
  is a real semi-analytic set. If the point $0 \in Z$ was in 
  $C^-$, then we can choose a real 
  analytic path $\gamma : [0, t_1) \to C^-$ with $\gamma(0) = 0$ and 
  $\gamma(t) \notin C^- \cap L$ for all $0<t< t_1$. 
  Comparing the integrals 
  \[
    \lambda(\gamma(t)) = \int_0^t \D \lambda(\gamma(\tau)) \dot\gamma(\tau) \D \tau 
    \quad \textnormal{ and } \quad
    \rho(\gamma(t)) = \int_0^t \D \rho(\gamma(\tau)) \dot\gamma(\tau) \D \tau
  \]
  we see that 
  \[
    a(\tau) = \frac{\D \lambda(\gamma(\tau))\dot \gamma(\tau)}{\D \rho(\gamma(\tau))\dot\gamma(\tau)}
  \]
  has a Laurent expansion at $\tau = 0$ and hence either $a(\tau)$ or $1/a(\tau)$ extends 
  to a real analytic function $\alpha$ on $[0,t_1)$. 
  Since both $\lambda$ and $\rho$ only assume non-negative values, the leading 
  coefficient of this function must be positive -- a contradiction to 
  $a<0$ in the construction of $C^-$.

  We can now finish the proof of $ii)$. Suppose there exists a point 
  $p \in \partial \overline Z_0$, which is the limit point of 
  a sequence $(p_n)_{n\in \NN}$ of relative boundary critical points 
  in the open stratum $S_0$. Then at each point $p_n$ we have 
  \[
    \D \lambda(p_n)|_{T_{p_n}S_0} = a_n \cdot \D \rho(p_n)|_{T_{p_n}S_0}
  \]
  with $a_n\leq 0$.
  Another application of Lemma \ref{lem:technicalSublemma} shows that in 
  this case we would have 
  \[
    \D \lambda(p)|_{T_p S_\alpha} = a \cdot \D \rho(p)|_{T_p S_\alpha},
  \]
  $S_\alpha$ the stratum containing $p$ and $a = \lim_{n\to \infty} a_n \leq 0$.
  If $a < 0$ this is a contradiction to (\ref{eqn:LinearDependenceEquation}). 
  If $a = 0$, then $\D \lambda(p) = 0$. But then $p\in L$ and we saw in 
  $i)$ that $p$ cannot be a limit of relative boundary critical points.

  We deduce that there exists an open neighborhood $V$ of $\partial \overline Z_0$ 
  in $Z$ such that no point in $V$ is a relative inward boundary critical point.
  Choose $\delta>0$ small enough such that $\partial \overline Z \cap h^{-1}(\Delta_\delta)$
  is contained in $V$. 
  
\end{proof}

\begin{proof}{(of Theorem \ref{thm:InductionStepConnectivity})}
  Fix an $\varepsilon>0$ and choose $\delta>0$ as in Lemma 
  \ref{lem:OnlyOutwardCriticalPoints}. Choose $t \in \Delta_\delta$, $t\neq 0$ 
  and consider $\lambda = |l|^2$ on $\overline Z_t$.

  Taking a slight perturbation $\tilde l$ 
  of the equation $l$, we may assume that 
  $\tilde \lambda = |\tilde l|^2$ is a Morse function on the manifold with boundary 
  $\overline Z_t$, see \cite{GoreskyMacPherson88}[Chapter 1, Example 2.2.4]. 
  We may also choose a regular value $l_0$ of $\tilde l$ on $\overline Z_t$ 
  close to $0$ and recenter $\tilde l$ at $l_0$ so that 
  $\overline Z_t^{\pitchfork}$ is smooth, while retaining the Morse properties. 
  Furthermore, since all the boundary critical points of $\lambda$ on 
  $\overline Z_t$ were outward and $\tilde \lambda$ can be chosen arbitrarily 
  close to $\lambda$ in the Whitney topology of smooth functions, 
  also the boundary critical points of $\tilde \lambda$ are outward. 
  With these choices we replace $l$ by $\tilde l$ and $\lambda$ by 
  $\tilde \lambda$ in what follows. 

  Since $l : \overline Z_t \to \CC$ is proper, the set of critical values 
  is closed and we find $\eta>0$ such that 
  \[
    l\colon l^{-1}(\overline \Delta_\eta) \cap \overline Z_t \to \overline \Delta_\eta
  \]
  is a trivial bundle 
  $\lambda^{-1}( [0,\eta] ) \cap \overline Z_t \cong 
  \overline Z_t^\pitchfork \times \Delta_\eta$. 
  For any $c \in [\eta,\infty)$ denote $\overline Z_t \cap \lambda^{-1}([0,c])$
  by $\overline Z_t^{\leq c}$.

  As remarked earlier, the real valued function $\lambda$ has real Morse 
  singularities of index $n = \dim_\CC \overline Z_t$ at each of 
  the complex Morse points of $l$ in the interior of $\overline Z_t$. 
  Consequently, up to homotopy, we attach an $n$-cell to 
  $\overline Z_t^{\leq c}$ each time $c$ passes through one of the 
  critical values of these points. 
  On the other hand, whenever $c$ crosses a critical value of a
  boundary critical point $p \in \partial \overline Z_t$, the 
  homotopy type of $\overline Z_t^{\leq c}$ does not change 
  according to Lemma \ref{lem:RelativeMorseTheory}, because 
  $p$ is an outward critical point.
  This finishes the proof.
\end{proof}

\begin{corollary}
  Suppose $h : (Z,0) \to (\CC,0)$ as in Theorem \ref{thm:InductionStepConnectivity} 
  is the smoothing of a complete intersection singularity of dimension $n$. 
  Then the associated 
  Milnor fiber $\overline Z_t$ has the homotopy type of a bouquet of spheres 
  of real dimension in the range from $n-d$ to $n$. In the extremal case 
  $n = d$ this means we allow finitely many components.
  \label{cor:KatoMatsumotoGeneralization}
\end{corollary}

\begin{proof}
  We proceed by induction on the dimension $d$ of the singular locus of $(Z_0,0)$.
  If $d=0$, then $(Z_0,0)$ is an isolated singularity and this is nothing but 
  Hamm's result \cite{Hamm72}. 
  In the particular case where $n = d = 0$, the map $h \colon (Z,0) \to (\Delta,0)$ 
  is a branched covering and $\overline Z_t$ consists of finitely many points.

  Now suppose the theorem has been established for $d$ in the range 
  $0 \leq d < k$ for some $k$ and we are 
  given a nonisolated singularity $(Z_0,0)$ with singular locus of dimension $k$.
  According to Theorem \ref{thm:InductionStepConnectivity} the space 
  $\overline Z_t$ has the homotopy type of $\overline Z_t^{\pitchfork}$ with 
  attached handles of real dimension $n$. Since $\overline Z_t^{\pitchfork}$ is 
  the Milnor fiber of a non-isolated complete intersection 
  singularity with singular locus of dimension 
  $k-1$, the claim follows. 
\end{proof}

\begin{remark}
  Note that Corollary \ref{cor:KatoMatsumotoGeneralization} gives bounds 
  on the homology of nearby fibers $\overline Z_t$ in \textit{any} smoothing of 
  a given complete intersection singularity $(Z_0,0)$, while the detailed topology of 
  these fibers $\overline Z_t$ will depend 
  on the choice of such a smoothing, i.e. on the embedding of $(Z_0,0)$ 
  in the family 
  \[
    (Z_0,0) \hookrightarrow (Z,0) \overset{h}{\longrightarrow} (\CC,0).
  \]
  Once this family has been chosen, we find ourselves in the 
  convenient setting of a hypersurface singularity given by the function 
  $h$ 
  on a singular ambient space $(Z,0)$ and 
  it makes sense to speak of the Milnor fibration of $h$ on $(Z,0)$ as 
  stated in Corollary \ref{cor:KatoMatsumotoGeneralization}.
  We may then apply 
  Theorems \ref{thm:FibrationTheorem}, \ref{thm:LeHyperplaneFibration}, and 
  \ref{thm:InductionStepConnectivity} in the provided setting.
\end{remark}

\section{Alternating Homology}

In this section we finally proof our Main Theorem \ref{thm:MainTheorem}.
A key ingredient will be a spectral sequence coming from the alternating
homology of the multiple point spaces, which we now recall. The original
construction of \cite{GoryunovMond93} uses cohomology. In \cite{Goryunov95} 
there is a version for homology. See also \cite{Mond16} and the recent 
preprint \cite{CisnerosMond18}.
The homological and the cohomological versions are dual to each other when using
coefficients in a field ($\QQ$ in our case), so we have chosen homology,
as it is more visual. We restrict our exposition to the objects which are
needed for the proof of our main theorem. 

\subsection{The spectral sequence for the homology of the image}
Given a finite map $f\colon X\to Y$, let $C_q(D^k)$ be the module of
$q$-dimensional singular chains on the multiple point space $D^k$,
with coefficients in $\QQ$. Recall that, for convenience, this includes
the case of $X=D^1$. For $k\geq 2$, the natural action of the symmetric
group $S_k$ on $D^k$ by permutation of the factors of $D^k\subseteq X^k$
induces a linear action on $C_q(D^k)$. The submodule of $k$-dimensional
alternating chains of $C_q(D^k)$ is defined as \[
  C_q^{\rm Alt}(D^k)= \left\{h\in C_q(D^k)\mid \sigma * h={\rm
  sign}(\sigma)\cdot h, \text{ for all }\sigma\in S_k \right\}.
\] For $k=1$, we set  $C_q^{\rm Alt}(D^1)=C_q(D^1)$.

We also have two morphisms from each $C_q(D^k)$: the usual boundary
operator \[
  \partial\colon C_q(D^k)\to C_{q-1}(D^k),
\] and the morphism \[
  \varphi \colon C_q(D^k)\to C_q(D^{k-1}),
\] given by $\varphi =\sum_{i=1}^{k}(-1)^i(\epsilon^{i,k})^\sharp,$
where $(\epsilon^{i,k})^\sharp$ are induced on chains by the maps
$\epsilon^{i,k}\colon D^{k}\to D^{k-1}$.

It is easy to see that $\partial$ and $\varphi$ take $C_q^{\rm
Alt}(D^k)$ to $C_{q-1}^{\rm Alt}(D^k)$ and $C_q^{\rm Alt}(D^{k-1})$,
respectively. These restricted morphisms give rise to a double complex
$$ \xymatrix{
 \vdots \ar[d]^{\partial}&\vdots \ar[d]^{\partial}&\vdots
 \ar[d]^{\partial} & \\
C_2^{\rm Alt}(D^{1})	  \ar[d]^{\partial}	& C_2^{\rm Alt}(D^{2})
\ar[l]_{\varphi}  \ar[d]^{\partial}	& C_2^{\rm Alt}(D^{3})
\ar[l]_{\varphi}  \ar[d]^{\partial}	 &\ar[l]_{\varphi} \cdots\\
C_1^{\rm Alt}(D^{1})	  \ar[d]^{\partial}	& C_1^{\rm Alt}(D^{2})
\ar[l]^{\varphi}  \ar[d]^{\partial}	& C_1^{\rm Alt}(D^{3})
\ar[l]^{\varphi}  \ar[d]^{\partial}	 &\ar[l]^{\varphi} \cdots\\
C_0^{\rm Alt}(D^{1})		& C_0^{\rm Alt}(D^{2})	  \ar[l]^{\varphi}
& C_0^{\rm Alt}(D^{3})	 \ar[l]^{\varphi}	  &\cdots \ar[l]^{\varphi}
}$$

Associated to this double complex there is a spectral sequence
$E_{\bullet, \bullet}$, whose first
page has as entries the homology groups
\[
  E_{p,q}^1 = H_{q}^{\Alt} (D^{p+1}):= H_q\left(
  C_{\bullet}^{\Alt}(D^{p+1}),\partial \right)
\] 
of the column complexes. Since we are working over $\QQ$, we may form
the alternating projection operator 
\[
  \Alt : C_q(D^k) \to C_q^{\Alt} (D^k), \quad w \mapsto \frac{1}{k!}
  \sum_{\sigma \in S_k} \sign(\sigma) \cdot \sigma * w,
\] 
and, because $\Alt$ commutes with the boundary operator $\partial$,
we have a natural extension of $\Alt$ to the homology groups $H_q(D^k)$. 
Let $\Alt H_q(D^k)$ be subgroup of alternating homology groups -- 
the image of $\Alt$ in $H_q(D^k)$. Then we have isomorphisms
\begin{equation}
  \Alt H_q(D^k) \cong H_q^{\Alt}(D^k).
  \label{eqn:IdentificationAlternatingHomology}
\end{equation}
The fact that we can regard $H_q^{\Alt}(D^k)$ as a subobject of
$H_q(D^k)$ leads to a crucial observation: \\

\textit{
  $H_q^{\Alt}(D^k)$ is zero if $D^k$ has trivial homology in dimension
  $q$.
}\\

\begin{remark}
  The identification of $H^{\Alt}_q(D^k)$ with a subgroup of $\Alt H_q(D^k)$ 
  relies on the choice of the coefficients in a field of characteristic zero 
  so that the alternating operator $\Alt$ can be defined as an idempotent 
  operator. For multiple point spaces of $\mathcal A$-finite maps of corank $1$, 
  which are ICIS with a finite group action, Goryunov proved in \cite{Goryunov95} that 
  (\ref{eqn:IdentificationAlternatingHomology}) also holds over the 
  integers for their stabilizations. However, his proof relies on the singularities 
  being isolated, which can not be assumed to be the case in our setup.
\end{remark}

\noindent
The importance of the alternating homology becomes clear in the 
next theorem.

\begin{theorem}
  Let $f \colon X \to Y$ be a finite stable map with image $Z \subset Y$.
  Then the spectral sequence
  \[
    E^1_{p,q} = H^{\rm Alt}_q(D^{p+1}) \Longrightarrow H_{p+q}(Z)
  \]
  in alternating homology from the double complex above converges 
  to the homology of $Z$. 
  \label{thm:SpectralSequenceForAlternatingHomology}
\end{theorem}

\noindent
For different proofs see \cite{GoryunovMond93}, \cite{Goryunov95}, and
\cite{CisnerosMond18}.\\

While computing the spectral sequence explicitly might be quite involved,
we can give estimates for the vanishing of the rational homology of $Z$
from information about the first page.

\subsection{Proof of Theorem \ref{thm:MainTheorem}}
Let $f : (\CC^n,0) \to (\CC^N,0)$ be a finite, dimensionally correct
holomorphic map germ of corank $1$ with non-isolated instability of
dimension $d = \dim \Inst(f)$ and 
\[
  F = (f_u,u) : X \times \Delta  \to Y \times \Delta 
\] a representative of a topological stabilization thereof. Choose
a Milnor ball $B = \rho^{-1}( [0,\varepsilon] ) \subset X$ for this
representative according to Sections \ref{sec:FibrationTheorems} and
\ref{sec:InducedDeformationsOfMultiplePointSpaces} so that, possibly after
shrinking $\Delta $, we can pass to the associated proper representatives $h :
\overline Z \to \Delta $ and $u : \overline {D^k(F)} \to \Delta $.

From Theorem \ref{thm:MararMondCurvilinearStabilityCriterion}, Corollary
\ref{cor:CurvilinearMapsHaveCompleteIntersectionsAsMultPtSp}, and
Corollary \ref{cor:EstimateOnSingularLocusForMultPtSpc} we know that
the induced deformations in the multiple point spaces \[
  u : \overline {D^k(F)} \to \Delta 
\] are smoothings of complete intersection singularities $(D^k(f),0)
= (u^{-1}( \{0\} ),0)$ with singular locus of dimension $\leq d$.
From Corollary \ref{cor:KatoMatsumotoGeneralization} we know that the
associated Milnor fibers $\overline D^k(f_s)$ have the homotopy type of a
bouquet of spheres $S^q$ of dimensions $q$ within the range from $\dim
(D_0^k,0) - d$ to $\dim (D_0^k,0)$. Since the alternating homology
groups $H_q^{\Alt}(\overline D^k(f_s),\QQ)$ are subgroups of $H_q(\overline
D^k(f_s),\QQ)$, Equation (\ref{eqn:IdentificationAlternatingHomology}), 
this gives immediate bounds on the nonzero terms of the
first page of the spectral sequence from Theorem 
\ref{thm:SpectralSequenceForAlternatingHomology}.

Observe that, as long as the inequality 
\[
  k<\left\lceil \frac{N-d}{N-n} \right\rceil
\] 
is satisfied, the singular locus of $(D^k(f),0)$ is
of strictly smaller dimension than $\dim (D^k(f),0)$, and hence
all the $\overline D^k(f_s)$ are connected. For any $\sigma\in S_k$, the points $\sigma(x)$ and $x$ can be connected by a path $\gamma$, so that $[x]-\sigma_\#[x]=\partial \gamma$. This implies that the action of $S_k$ is trivial on the degree zero homology, and hence that the alternating homology vanishes in degree zero.
On the other
hand, within the range 
\[
  \left\lceil \frac{N-d}{N-n} \right\rceil \leq k \leq
  \left\lfloor \frac{N}{N-n} \right\rfloor
\] 
the $\overline D^k(f_s)$ may have several
components, which allows alternating homology in degree zero.

The result follows from the computation of the spectral sequence: Each
homology group $H_r(\overline Z_s)$ has a filtration, whose successive
quotients are isomorphic to quotients of subspaces of entries in the
$r$-th anti-diagonal in (\ref{eqn:SpectralSequenceFirstPage}). The
given bounds on nonzero homology groups come from the first and the
last possible intersection of the anti-diagonals with the area of possibly
nonzero terms.

\begin{equation}
  \begin{array}{ccccccccccccccccccccccccccccc}
 \multicolumn{1}{c|}{}	&
    \vdots & \vdots & \vdots & \vdots & & & & & \\
  \multicolumn{1}{c|}{2n - N +1} &
    \times & \times & \times & \times & & & & & \\
  \multicolumn{1}{c|}{2n - N} &
    \times & \bullet & \times & \times & & & & & \\
  \multicolumn{1}{c|}{	  \vdots} &
    \vdots & \vdots & \vdots & \vdots & & \vdots & \vdots & \vdots & \\
  \multicolumn{1}{c|}{	  2n - N - d + 1} &
    \times & \bullet & \bullet & \times & & \times & \times & \times &
    \cdots \\
  \multicolumn{1}{c|}{	  2n - N - d }&
    \times & \bullet & \vdots & \times & & \times & \times & \times &
    \cdots \\
   \multicolumn{1}{c|}{   2n - N - d -1} &
    \times & \times & \bullet & \bullet & & \times & \times & \times &
    \cdots \\
  \multicolumn{1}{c|}{	  \vdots }&
    \vdots & \vdots & \bullet & \vdots & & \times & \times & \times &
    \cdots \\
  \multicolumn{1}{c|}{	  4} &
    \times & \times & \times & \bullet & & \bullet & \times & \times &
    \cdots \\
  \multicolumn{1}{c|}{	  3} &
    \times & \times & \vdots & \bullet & & \vdots & \times & \times &
    \cdots \\
  \multicolumn{1}{c|}{	  2} &
    \times & \times & \times & \times & & \bullet & \bullet & \times &
    \cdots \\
  \multicolumn{1}{c|}{	  1} &
    \times & \times & \times & \vdots & & \bullet & \vdots & \times &
    \cdots \\
  \multicolumn{1}{c|}{	  0} &
    \bullet & \times & \times & \times & & \times & \bullet & \bullet &
    \times \\
\hline	\multicolumn{1}{c|}{}
    & 1 & 2 & 3 & 4 & \cdots & \left\lceil \frac{N-d}{N-n} \right\rceil -1
    & \cdots & \left\lfloor \frac{N}{N-n} \right\rfloor &
  \end{array}
  \label{eqn:SpectralSequenceFirstPage}
\end{equation}

\noindent
The vertical index in this diagram is $q$, the horizontal one $k$. 
We write $\times$ for a zero term and $\bullet$ for a possibly nonzero term.
\qed
\subsection{Revisiting an example} Consider the deformations $f_{a,b,c}$
of $B_\infty$, given by \[(x,u)\mapsto (x^2, u^2 x+a x+bx^3+c x^5,u),\]
as in Example \ref{exDeformationsBInfty}.  The divided differences
of the function $f_2(x,u)=x^2$ are \[f_2[x,x',u]=x+x',\text{ and }
f_2[x,x',x'',u]=1.\] Therefore $D^r(f)=\emptyset$, for $r\geq 3$. The
coordinate $x'$ may be eliminated, replacing it by $-x$. This takes
$D^2(f_{a,b,c})$ isomorphically to the plane curve \[u^2+b x^2+c x^4=-a,\]
and turns the action of the generator of $S_2$ on $D^2(f_{a,b,c})$ into
$x\mapsto -x$.	Writing \[A_0=H_0^{\Alt}(D^{2}(f_{a,b,c}))\quad \text{
and }\quad A_1=H_1^{\Alt}(D^{2}(f_{a,b,c})),\] the first page of the
spectral sequence looks as follows: 
\begin{equation*}
  \begin{array}{rcccccccc} \multicolumn{1}{c|}{\vdots}	&
     & & &
    \\
  \multicolumn{1}{c|}{\times}  &
    \vdots & \vdots & \vdots & \\
     \multicolumn{1}{c|}{H_2(Y)}&
    \times & \times & \times & \cdots \\
     \multicolumn{1}{c|}{H_1(Y)}&
    \times & A_1 & \times & \cdots \\
     \multicolumn{1}{c|}{H_0(Y)}&
    \QQ & A_0 & \times &
  \cdots
   \\
\cline{1-5}	\multicolumn{1}{c|}{}&
     X	&D^2 &
    \emptyset & \cdots &
	\\
  \end{array}
\end{equation*}

Now we study the double point spaces of the singularities considered
in Example \ref{exDeformationsBInfty}. These are depicted in Figure
\ref{fig:BifDgrmBinfty}, placed as in Figure \ref{fig:Example1}.
\begin{figure}[h]
  \centering \includegraphics[scale=1.2,trim=0.5cm 0 0 0cm]{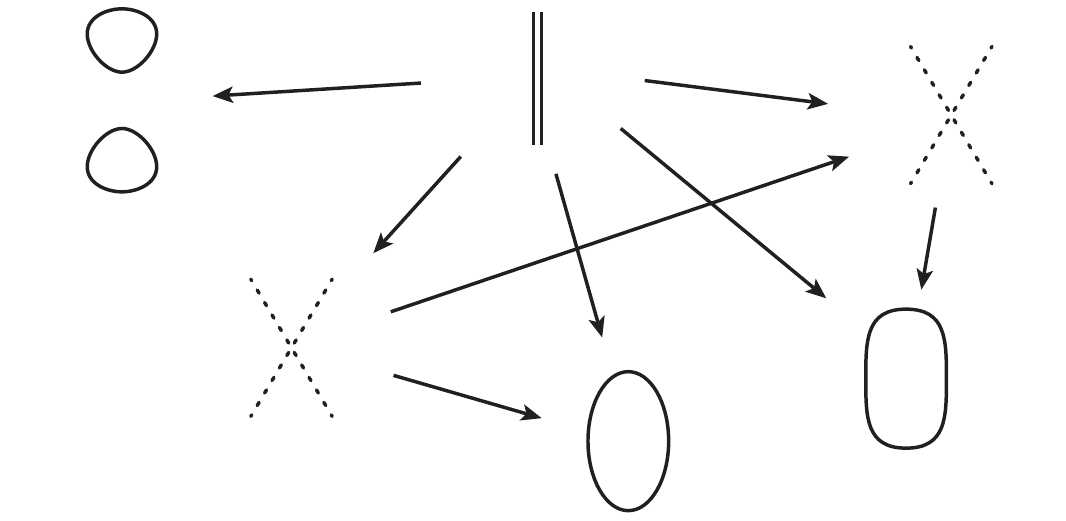}
  \caption{The spaces $D^2(f)$ of some singularities adjacent to
  $B_\infty$. Dashed lines correspond to phenomena which are not visible
  over the reals.} \label{fig:BifDgrmBinfty}
\end{figure}

Take the  map $f_{t/10,-t,t}$ on the top-left corner of the figure. Its
double point space has two connected components $C_1$ and $C_2$, each
containing a $1$-cycle. Each of these components is the image of the
other by the generator $\sigma\in S_2$. Consequently, for any point $a\in
C_1$ and any $1$-cycle generating $H_1(C_1)$, the cycles $a, \sigma a$  and $c,
\sigma c$ generate $H_0(D^2)$ and $H_1(D^2)$, respectively.  It follows
that $H_0^{\Alt}(D^2)$ and $H_1^{\Alt}(D^2)$ are generated, respectively,
by the non-zero cycles $a-\sigma a$ and $c-\sigma c$. Hence $A_0\cong \QQ$
and $A_1\cong \QQ$ are responsible for the homology in dimensions one and
two we observed.

Now look at the double spaces of the other maps. Since they are
connected, their homology in dimension 0 is generated by any point,
say $a\in D^2$. Since $a$ and $\sigma a$ are contained in the same
connected component, they represent the same homology class. This means
that $H_0(D^2)$ is $S_2$-invariant, hence $A_0=0$ for all these maps.

For 1-dimensional homology, observe that $D^2$ is contractible for
$B_1$ (top-right), while its stable perturbation (bottom-right) contains
a 1-dimensional cycle that switches orientation with the action of
$S_2$. This cycle is the generator of $A_1=H_1^{\Alt}(D^2)$ for the
mentioned stabilization. Similar considerations can be made for the
remaining singularities in the example. Recall that the difference between the two stable perturbations relies on the choice of good representatives, not depicted here.

The argument we gave to study $A_0=H_0^{\Alt}(D^{2}(f_{a,b,c}))$ for
the left-top example generalizes immediately to the following result:
\begin{proposition} A stable perturbation of a map germ $\CC^n\to \CC^{m}$,
with $2\leq n<m$ has non-trivial 1 dimensional homology if and only
if $D^2$ has at least one connected component not fixed by $S_2$.
\end{proposition}

\bibliography{sources}

\bibliographystyle{amsplain}
\end{document}